\documentclass[12pt]{article}
\usepackage[amsmath]{e-jc}

\makeatletter
\renewcommand{\ps@plain}{%
  \renewcommand{\@oddhead}{}%
  \renewcommand{\@evenhead}{}%
  \renewcommand{\@oddfoot}{\hfil\thepage\hfil}%
  \renewcommand{\@evenfoot}{\hfil\thepage\hfil}%
}
\makeatother

\usepackage{graphicx}
\usepackage{mathtools}
\usepackage{comment}
\usepackage{float}
\usepackage{booktabs}
\usepackage{siunitx}
\usepackage{algorithm}
\usepackage{algpseudocode}
\usepackage{setspace}
\usepackage{ytableau}
\usepackage[numbers,sort&compress]{natbib}

\hypersetup{
    colorlinks=true,
    citecolor=blue,
    linkcolor=blue,
    urlcolor=blue
}

\providecommand{\kmk}[1]{}

\title{Recursive Record Filtering and Longest Decreasing Subsequences}
\author{Jackson Zariski\authornote{1}
\and
Kaitlin Kratter\authornote{2}
}

\authortext{1}{Department of Applied Mathematics, University of Arizona
   (\email{jzariski@arizona.edu}), }
\authortext{2}{Steward Observatory, University of Arizona}

\dateline{TBD}{TBD}{TBD}
\Copyright{The author.}

\begin{document}

\maketitle
\begin{abstract}
We consider a recursive record-filtering procedure, which we informally call Disappear-Sort, acting as a sort of parallel to traditional patience sorting. Let $D_n$ denote the number of passes required to eliminate a sequence of length $n$ sampled as i.i.d.\ copies of a continuous random variable, where each pass retains the left-to-right records and applies the same rule recursively to the remaining entries. For the non-resampling procedure, we associate to a permutation $p_n\in S_n$ a natural poset and show that the recursive Disappear-Sort layers form an antichain decomposition of this poset. This provides an order-theoretic interpretation of the procedure and identifies the total number of passes with $L(p_n)$, the length of the longest decreasing subsequence of $p_n$. Equivalently, after reversing the comparison direction, the Disappear-Sort layers coincide with the piles arising in classical patience sorting.

For a uniformly random permutation, the pass count therefore has the same distribution as the first-column length of the tableau shape produced by the Robinson--Schensted correspondence. We use this classical connection to express $\mathbb{E}[D_n]$ as a sum over partitions and standard Young tableaux. Established results on Plancherel-random Young diagrams then imply $\mathbb{E}[D_n]\sim 2\sqrt{n}$, with fluctuations on the $n^{1/6}$ scale governed by the Tracy--Widom distribution. We also consider a resampling variant in which the nonrecord entries are replaced after each pass by a fresh independent sample of the same size, and derive an exact recurrence for its expected number of passes involving the unsigned Stirling numbers of the first kind. We conclude with an $O(n\log n)$ implementation for computing the non-resampling pass count.
\end{abstract}

%%%%%%%%%%%%%%%%%%%%%%%%%%%%%%%%%%%%%%%%%%%%%%%%%%%
\noindent The inspiration for this work comes from the sorting algorithm informally known as \emph{Disappear-Sort} (DS for short, sometimes also referred to as \emph{Stalin-Sort}) popular in some instructional examples in computer science. A single DS pass begins at the first element of the list and proceeds from left to right, deleting each entry that is smaller than the current running maximum. Whenever a larger entry is encountered, it becomes the new running maximum, i.e. the next record. A DS pass therefore produces an increasing list, but generally at the cost of substantial data loss. For this reason, it is usually regarded as a pedagogical example rather than as a practical sorting algorithm.

However, an extension of this sorting method can be made if we choose to keep the ordering of the discarded values and then construct a new list. Sorting this list would result in more discards, and we could continue recursively until we are left with no discards at all. We call this the \emph{DS procedure}. This leads to the following question: given $n$ i.i.d. samples from a continuous random variable $X$, what is the expected number of recursive Disappear-Sort passes needed to eliminate the list? There is extensive work in record theory on the number of records in a random sample, but we are not aware of prior work on the number of records arising in the successive discard lists generated by this recursive procedure. The leader-election procedure studied by \citet{electionrecord} is related, but it iterates on records rather than on the conditional discard lists considered in this procedure. \citet{aldous_diaconis} and \citet{burstein2006combinatorics} apply similar techniques to traditional patience-sorting, though we place it in the exact context of the instructional sorting algorithm. Below is an example of the full recursive, non-resampling DS procedure on the example list $[2,5,3,9,6,4]$: \begin{align*}
    [2,5,3,9,6,4] &\to [2,5,9], ~~\mathrm{Discards} ~[3,6,4] \\
    &\to [2,5,9], [3,6],~~\mathrm{Discards} ~[4] \\
    &\to [2,5,9], [3,6], [4]
\end{align*}
Thus, the recursive DS procedure decomposes the original list into three increasing sublists. Although this example is discrete, the analysis below assumes samples drawn from a continuous distribution. In addition, another version of this method, called the Resampling Disappear-Sort procedure (RDS) changes the algorithm so that entries in the discard list are resampled independently from the original distribution. Though this variant is simpler---because deeper passes depend only on the cardinality of the discard set rather than on its values---it yields an exact recurrence and provides a useful companion model.

We begin with an overview of record theory and its connection to Disappear-Sort, both with and without resampling. We then derive exact formulas for the first moments in both procedures. After presenting Monte Carlo comparisons, we turn to the asymptotic behavior of the non-resampling model.

\section{An Overview of Record Theory and Expected Single-Pass Records}\label{sec:overview}
Consider the probability space $(\Omega, \mathcal{F}, \mathbb{P})$ and the continuous random variable $X:\Omega\to\mathbb{R}$ with cumulative distribution $F$. We define the initial list $M_0$ such that \begin{equation}\label{eq:M}
  M_0 = [X_1, X_2, \cdots, X_n]
\end{equation}
where each $X_i$ is independent and identically distributed with the same distribution as $X$. 
 Define the \emph{record indices} $T_{M_0}=\{t_{M_0}^{(i)}\}$ of $M_0$ together with the corresponding record values $R_{M_0} = \{r_{M_0}^{(i)}\}$ by \begin{equation}
    t_{M_0}^{(0)} = 1, ~~~ t_{M_0}^{(i+1)} = \min\{j: j > t_{M_0}^{(i)}, X_j > X_{t_{M_0}^{(i)}}\}
\end{equation} 
\begin{equation}
    r_{M_0}^{(i)} = X_{t_{M_0}^{(i)}}
\end{equation}
Because $X$ has a continuous distribution, ties occur with probability zero. We restrict attention to upper records (the lower-record version is analogous) and DS therefore produces an increasing list. Now, we define the indicator function 
\begin{equation}
    I_{M_0}(i) = \begin{cases}
        1, & i\in T_{M_0} \\
        0, & \mathrm{otherwise}
    \end{cases}
\end{equation}

Hence, we seek the probability $\mathbb{P}(I_{M_0}(i)=1)$, or the probability that the value at index $i$ in list $M_0$ is a record. For the traditional proof, we follow the logic from \citet{french_renyi}.

\begin{lemma}[R\'enyi]\label{lem:dontmatter}
Let $X_1,\cdots,X_i$ be i.i.d. with a continuous distribution. Then the record probability $\mathbb{P}(I_{M_0}(i)=1)=\frac{1}{i}$. 
\end{lemma}

\begin{proof}
For simplicity, suppose first that $F$ is absolutely continuous with density $f$; the general continuous case follows from the symmetry of the ranks. Recalling the independence of the variables, we see that: \begin{align*}
  \mathbb{P}(I_{M_0}(i)=1)&=\mathbb{P}(X_i=\max\{X_1,X_2,\cdots X_i\}) \\
  &=\int_{-\infty}^{\infty}\mathbb{P}(\max\{X_1,X_2,\cdots,X_{i-1}\}<x)f(x)dx \\
  &=\int_{-\infty}^{\infty}F(x)^{i-1}f(x)dx
\end{align*}
Substituting in $u=F(x)$ and $du=f(x)dx$ we obtain \begin{align}
    \int_{-\infty}^{\infty}F(x)^{i-1}f(x)dx&=\int_0^1u^{i-1}du \notag\\
    &=\left[\frac{u^i}{i}\right]_0^1 \notag\\
    &=\frac{1}{i}
\end{align}
Therefore $\mathbb{P}(I_{M_0}(i)=1)=\frac{1}{i}$. In particular, the probability that the element at index $i$ is a record depends only on $i$, not on the underlying continuous distribution $F$. 
\end{proof}
This classical result of R\'enyi allows us to determine the expected number of records in the initial list $M_0$ of size $n$. Now, define $N^i_{M_0}$ to be the random variable corresponding to the number of records in $M_0$ up to index $i$, so \begin{equation}
    N^i_{M_0}=\sum_{k=1}^i I_{M_0}(k)
\end{equation}
By linearity of expectation, we obtain: \begin{align} \label{eq:Hi}
    \mathbb{E}[N^i_{M_0}]&=\mathbb{E}\left[\sum_{k=1}^i I_{M_0}(k)\right] \notag\\
    &=\sum_{k=1}^i \mathbb{E}[I_{M_0}(k)] \notag\\
    &=\sum_{k=1}^i\frac{1}{k}=H_i 
\end{align}
where $H_i$ is the $i$th harmonic number. So the expected number of records in a list up to index $i$ is simply $H_i$ \citep{french_renyi}.
\section{First Moment of the Expected Passes in Disappear-Sort}
In the previous section, we considered only the initial list of values $M_0$. However, keeping the information provided in the ordered discard (non-record) set allows us to explore what we call the DS procedure. Let $N_{M_t}:=|R_{M_t}|$ be the number of records in the list $M_t$. If $|M_0|=n$, then recursively we define in the \textit{non-resampling} (DS) procedure \begin{equation}\label{eq:nonresampling}
    M_{t+1} = M_t \setminus R_{M_t}~~~ (m_{t_i}\in M_t :m_{t_i}\notin R_{M_t})
    \end{equation}
    \begin{equation}\label{eq:cardinality}
    |M_{t+1}| = |M_t| - N_{M_t}
\end{equation}
Note \eqref{eq:nonresampling} is an order-preserving list-minus operator. The procedure \textit{with resampling} still has \eqref{eq:cardinality} hold, but \eqref{eq:nonresampling} is ignored and $M_{t+1}$ is resampled independently with length $|M_t|-N_{M_t}$.
The resampling procedure thus does not depend on the values of the records left in the previous list, only the cardinality of the discard set. Now, define the quantity $D(M_0)$ to be the number of passes required to eliminate the originally sampled list $M_0$. Equivalently, this is the number of recursive passes required until no entries remain. For instance, if the list $[2,5,3,9,6,4]$ from our introduction was $M_0$, then it would result in $D(M_0)=3$ (again only in the non-resampling case). The goal of this section is to derive exact formulas for $\mathbb{E}[D_n]$, the expected number of passes in the two variants of the recursive DS procedure applied to a random list of size $n$.

\subsection{First Moment of the RDS Procedure}\label{sec:fmomresamp}
 To begin, we consider the recursive RDS procedure with resampling, as its first moment is more straightforward to derive as an exact recurrence. Analytically $D_n$ is a random variable where:
\begin{equation}
    D_n := \min\{t \ge 0 : |M_t| = 0,\ |M_0| = n\}
\end{equation}
and $M_0$ is the initial list described in \eqref{eq:M}. Additionally, we take $N_n$ to be the random variable associated to the number of records in a list of size $n$. In \eqref{eq:Hi} we found a closed-form expression to calculate $\mathbb{E}[N_n]$ directly. First, given a list of size $n$ we need to find $\mathbb{P}(N_n=r)$. Note that a sampled list $M_0$ of size $n$ can be permuted in $n!$ possible ways with uniform likelihood since we are drawing from a continuous distribution. Consider the following classical lemma.
We make the usual reduction from i.i.d.\ continuous samples to uniform random permutations.

\begin{lemma}\label{lem:uniform-ranking}
Let $X_1,\dots,X_n$ be i.i.d.\ real-valued random variables with a continuous
distribution.  Define a random permutation $p\in S_n$ (the symmetric group
on $\{1,\dots,n\}$) by the requirement that
\[
X_{p(1)} < X_{p(2)} < \cdots < X_{p(n)},
\]
whenever all $X_i$ are distinct.  Then:
\begin{enumerate}
\item $\mathbb{P}(X_i \neq X_j \text{ for all } i\neq j) = 1$, so $p$ is
well-defined almost surely; and
\item For every $p^*\in S_n$,
\[
\mathbb{P}(p = p^*) = \frac{1}{n!}.
\]
In particular, $p$ is uniformly distributed on $S_n$.
\end{enumerate}
\end{lemma}

Because the $X_i$ are i.i.d. continuous, the permutation of the sample $(X_1,\cdots,X_n)$ is uniform on the finite symmetric group $S_n$ almost surely. For an i.i.d.\ sample from a continuous distribution, ties occur with
probability zero, and the induced rank permutation is uniformly
distributed on $S_n$, i.e. $\mathbb{P}(p = p^*) = \frac{1}{n!}$. This is a standard fact from order statistics and rank theory \citep{david_nagaraja}.

Now we derive the first moment of the resampling procedure.

\begin{theorem}
    Let $d_n=\mathbb{E}[D_n]$ represent the expected value of the number of passes needed in the DS resampling procedure to eliminate a randomly sampled list of size $n$ from a continuous, i.i.d. distribution. Then $$
        d_n=1 + \sum_{r=1}^nd_{n-r}\frac{c(n,r)}{n!}
    $$ 
\end{theorem}
\begin{proof}

From R\'enyi's work on permutations, we know that the number of permutations with $1\leq r\leq n$ records is equal to the unsigned Stirling number $c(n,r)$ of the first kind \citep{french_renyi}. Recall from Lemma \ref{lem:dontmatter} without loss of generality we can assume that we are working with any continuous i.i.d. distribution. So, combined with Lemma \ref{lem:uniform-ranking} we obtain, \begin{equation}
    \mathbb{P}(N_n=r)=\frac{c(n,r)}{n!}
\end{equation}
Note that \begin{equation}
    d_n = 1 + \mathbb{E}[D_{n - N_n}]
\end{equation}
from the fact that our subsequent list $M_1$ is resampled with length $n-N_n$. Conditioning on $N_n$, we obtain
\begin{align*}
    d_n &= 1 + \mathbb{E}[D_{n - N_n}] \\
    &=1 + \sum_r\left[\mathbb{E}[D_{n - N_n} ~|~N_n=r]\cdot\mathbb{P}(N_n=r)\right] \\
    &=1 + \sum_r\mathbb{E}[D_{n-r}]\frac{c(n,r)}{n!} \\
    &=1 + \sum_rd_{n-r}\frac{c(n,r)}{n!} \\
\end{align*}
Note that our second-to-last equality above comes because we are working with the \emph{resampling} procedure\begin{equation}
\mathbb{E}[D_{n-N_n} ~|~N_n=r]=\mathbb{E}[D_{n-r}]
\end{equation}
since only the cardinality of the previous discard set is preserved and individual record values are independent through layers. Therefore the expected number of passes in the resampling RDS procedure is \begin{equation}
    d_n=\mathbb{E}[D_n]=1 + \sum_{r=1}^nd_{n-r}\frac{c(n,r)}{n!}
\end{equation}
\end{proof}
We illustrate this exact recurrence by a Monte-Carlo simulation in Figure \ref{fig:resampleMC}. This recurrence can be evaluated dynamically, with initial conditions $d_0=0$ and $d_1=1$.
\begin{figure}
\begin{center}
\includegraphics[width=0.8\textwidth]
{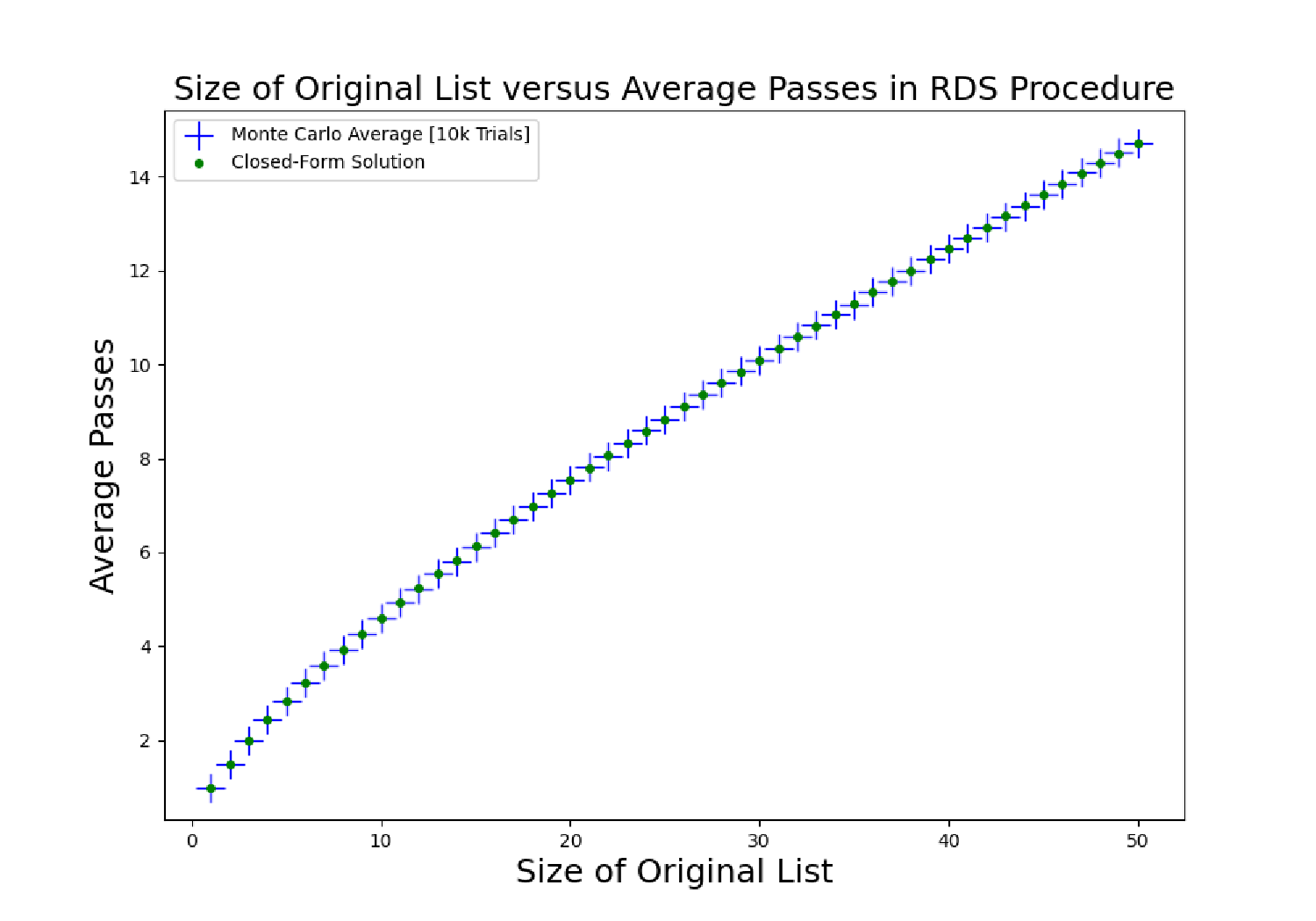}
\caption[Resampling Monte Carlo simulation]{A Monte Carlo simulation to visualize the exact recurrence for the RDS problem.}
\label{fig:resampleMC}
\end{center}
\end{figure}

\subsection{First Moment of the DS Procedure}
We now turn to the non-resampling procedure, which we abbreviate to DS. This work is adapted from traditional patience-sorting techniques, with initial comparisons to the RSK correspondence most notably from \citet{aldous_diaconis} and more recently \citet{burstein2006combinatorics}. From Section \ref{sec:overview} we know that the number of records in the originally sampled list does not depend on the underlying distribution, only on the assumption that the distribution is continuous. Accordingly, one could compute the first moment naively by enumerating all permutations of $n$ distinct values and applying the DS procedure to each of them. Formally, let $P=\{p_i\}$ denote the set of permutations of $M_0$ and define \[
D(p): P\to\mathbb{N}
\]
to be a function for the number of passes needed to complete a full Disappear-Sort procedure (eliminate the list). Then we see \begin{equation}
    \mathbb{E}[D_n] = \frac{1}{n!}\sum_{p_i}D(p_i)
\end{equation}
where $n$ is the length of the starting list in question. 
However, since this naive process scales with $\mathcal{O}(n!\cdot n\log n)$ time (see Section \ref{sec:ta}), it is not feasible for any large or even modest value of $n$. We therefore seek a more efficient exact formula.

First, fix a
permutation
\[
p = (p_1,\dots,p_n)
\]
of the values in the original list $M_0$ and define a partial order on the
index set $\{1,\dots,n\}$ by
\[
i \prec j \quad \Longleftrightarrow \quad i < j \ \text{and}\ p_i > p_j.
\]
Totally ordered subsets of this poset correspond exactly to decreasing
subsequences of $p$, so the height of this poset (the size of the largest
chain) is the length of the longest decreasing subsequence of $p$.

We now provide a few lemmas to make precise the relationship between this poset and the DS
procedure.

\begin{lemma}\label{lem:records-minimal}
Let $p=(p_1,\dots,p_n)$ and define $\prec$ as above.
Then the minimal elements of the poset $(\{1,\dots,n\},\prec)$ are exactly
the record indices of the sequence $p$, i.e.\ those $i$ for which
\[
p_i > \max\{p_j : j < i\}.
\]
where for $i=1$ the condition is interpreted as holding automatically.
\end{lemma}

\begin{proof}
Suppose first that $i$ is a record index.  Then by definition
$p_i > p_j$ for all $j<i$.  If by contradiction there were some $j\prec i$ ($i$ not minimal), we would have
$j<i$ and $p_j > p_i$, which contradicts the record property.  Thus no
$j\prec i$ and $i$ is minimal in the poset.

Conversely, suppose by contradiction $i$ is not a record index.  Then there exists some
$j<i$ with $p_j > p_i$.  By definition this means $j\prec i$, so $i$ is not
minimal.  Hence the minimal elements are precisely the record indices.
\end{proof}

In one pass of the DS procedure we keep exactly the record
indices of $p$, so by Lemma~\ref{lem:records-minimal} the set of indices
kept in the first pass is exactly the set of minimal elements of the poset.
The following lemma makes this connection recursive through subsequent passes in the procedure and also sets up an antichain paradigm we will use later.

\begin{lemma}\label{lem:DS-layers-antichain-partition}
Let $A_1$ be the set of indices kept after the first DS pass on $p$ (i.e.\
the record indices), and let $M_1$ be the list of discarded elements,
preserving their original order.  Apply DS again to $M_1$ and let $A_2$ be
the set of indices (in the original permutation) kept on this second pass.
Continue recursively, obtaining disjoint subsets
\[
A_1,A_2,\dots,A_{D(p)} \subseteq \{1,\dots,n\},
\]
where $D(p)$ is the total number of passes performed by DS on $p$.

Then:
\begin{enumerate}
\item each $A_t$ is an antichain in $(\{1,\dots,n\},\prec)$; and
\item the sets $A_1,\dots,A_{D(p)}$ form a partition of $\{1,\dots,n\}$.
\end{enumerate}
\end{lemma}

\begin{proof}
We first show that each $A_t$ is an antichain. Let $i<j$ be indices kept in the same DS pass. Since both indices are retained in that pass, the corresponding elements are records in the list entering that pass. If $i\prec j$, then $i<j$ and $p_i>p_j$, so the element indexed by $j$ would have an earlier larger element in that pass and hence could not be a record, a contradiction. On the other hand, $j\prec i$ is impossible because $j<i$ fails. Therefore no two elements of $A_t$ are comparable, so $A_t$ is an antichain.

We argue by induction on $t$ that each index belongs to exactly one of the
sets $A_t$, and that $A_t$ is precisely the set of minimal elements of the
induced subposet on the remaining indices. For $t=1$ this is Lemma~\ref{lem:records-minimal}.  Suppose the claim holds
for $t=1,\dots,T-1$, and let $P_T$ be the induced subposet obtained from
$(\{1,\dots,n\},\prec)$ by removing $A_1\cup\cdots\cup A_{T-1}$.  By the
definition of the recursive DS procedure, the $T$-th pass is applied
exactly to the list $M_{T-1}$ consisting of the discarded elements, in
their original order.  Applying the argument of
Lemma~\ref{lem:records-minimal} to $P_T$ shows that the indices kept in
this pass are precisely the minimal elements of $P_T$, and again form an
antichain.  These indices are by definition $A_T$.

Since at each step we remove all minimal elements of the current induced
subposet and never re-introduce any indices, the process terminates after a
finite number of steps, and $\{A_1,\dots,A_{D(p)}\}$ is a partition of
$\{1,\dots,n\}$ into antichains.
\end{proof}

We will use the theorem of \citet{mirsky} relating the height of a poset to the minimal number of antichains partitioning the poset.

\begin{theorem}[Mirsky]\label{thm:mirsky}
For any finite poset $P$, the minimal number of antichains whose union is
$P$ is equal to $h(P)$, the size of a longest chain.
\end{theorem}

This will be needed later for one direction in proving an equality for $D(p)$, along with the following lemma.

\begin{lemma}\label{lem:lengthofD(p)}
Let $p=(p_1,\dots,p_n)$ be a permutation and run the recursive DS procedure. Let
$D(p)$ denote the total number of passes until the procedure terminates, and define
\[
i \prec j \iff i<j \text{ and } p_i>p_j.
\]
Then there exists a chain
\[
i_1 \prec i_2 \prec \cdots \prec i_{D(p)}
\]
with $D(p)$ elements. In particular, if $L(p)$ denotes the size of the longest
$\prec$-chain, then
\[
L(p)\ge D(p).
\]
\end{lemma}

\begin{proof}
For each pass $t=1,\ldots,D(p)$, let $B_t$ denote the list entering pass $t$, so $B_1$ is the original list and $B_{t+1}$ is the discard list produced from $B_t$.
Choose $i_{D(p)}$ to be the original index of any element retained in the final pass.
Now proceed backward. For each $t=D(p),D(p)-1,\dots,2$, the element $p_{i_t}$
appears in $B_t$, so by definition it was discarded during pass $t-1$. Hence,
when $p_{i_t}$ was scanned in pass $t-1$, there existed an earlier record in that
pass whose value exceeded $p_{i_t}$. Choose one such record and let its original
index be $i_{t-1}$.

Then $i_{t-1}<i_t$ and $p_{i_{t-1}}>p_{i_t}$, so
\[
i_{t-1}\prec i_t.
\]
Repeating this construction yields a chain
\[
i_1 \prec i_2 \prec \cdots \prec i_{D(p)}.
\]
Therefore there exists a $\prec$-chain with $D(p)$ elements, and hence
\[
L(p)\ge D(p).
\]
\end{proof}

Applying Theorem~\ref{thm:mirsky} to the poset $(\{1,\dots,n\},\prec)$ together with Lemma~\ref{lem:DS-layers-antichain-partition}, we now prove the key equality underlying the exact formula for the DS expectation.

\begin{theorem}\label{thm:DS-LDS}
Let $p=(p_1,\dots,p_n)$ be a permutation and let $D(p)$ be the number of
passes performed by the non-resampling DS procedure on $p$.  Let $L(p)$
denote the length of the longest decreasing subsequence (chain) of $p$.  Then
\[
D(p) = L(p).
\]
\end{theorem}

\begin{proof}
By Lemma~\ref{lem:DS-layers-antichain-partition}, the sets
$A_1,\dots,A_{D(p)}$ produced by the recursive DS procedure form a partition
of the poset into $D(p)$ antichains. By Mirsky's Theorem (Theorem~\ref{thm:mirsky}), the minimal number of antichains among all partitions is equal to the height of the poset, or the size of the longest chain. 

Since the sets $A_1,\cdots A_{D(p)}$ form a partition into $D(p)$ antichains, Mirsky's theorem implies that 
\[
h(P)\leq D(p)
\]
where $P=(\{1,\cdots,n\},\prec)$. Because chains in $P$ are exactly decreasing subsequences of $p$, we have $h(P)=L(p)$, and therefore $L(p)\leq D(p)$.
On the other hand, Lemma \ref{lem:lengthofD(p)} gives $L(p)\geq D(p)$. Combining
the two inequalities yields $D(p)=L(p)$. Since chains in this poset are exactly the decreasing subsequences of $p$, this is precisely the length of the longest decreasing subsequence.
\end{proof}

In particular, for a uniform random permutation $p$ of length $n$ we have
\begin{equation}\label{eq:EK-L}
\mathbb{E}[D_n] = \mathbb{E}[L_n]
\end{equation}
where $L_n$ is the random variable corresponding to the longest decreasing subsequence of a uniform random permutation in $S_n$. As discussed in Section \ref{sec:ta}, this subsequence characterization yields an $O(n\log n)$ algorithm for computing the number of passes, without explicitly generating the successive DS layers.

\begin{figure}\label{fig:syt}
\begin{center}
\[
\scalebox{2.0}{
\begin{ytableau}
1 & 2 & 5 \\
3 & 4 \\
6
\end{ytableau}
}
\]
\caption{An example standard Young tableau of shape $(3,2,1)$, whose shape is a partition of 6.}
\end{center}
\end{figure}

With this relation in hand, we now invoke Schensted's theorem, which identifies the longest decreasing subsequence length with the first-column length of the tableau shape under the Robinson-Schensted correspondence \citep{robinson, LISLDS}.

\begin{theorem}[Schensted]\label{thm:schenthm}
Let $p\in S_n$ and let $(P,Q)$ be the pair of standard Young tableaux of the same shape $\lambda\vdash n$ associated with $p$ under the Robinson-Schensted correspondence. Then the length of the longest decreasing subsequence of $p$ is $\lambda'_1$, the length of the first column of $\lambda$. 
\end{theorem}

The Robinson-Schensted correspondence gives a bijection between $S_n$ and pairs $(P,Q)$ of standard Young tableaux of the same shape $\lambda \vdash n$. Therefore \[
D(p)=L(p)=\lambda'_1
\]
We take $f^{\lambda}$ to be the number of standard Young tableaux with shape $\lambda$. Then the RS bijection states that the number of possible permutations with shape $\lambda$ is $(f^{\lambda})^2$. Define the random variable $\Lambda_n$ to represent the shape of the tableau generated by a uniformly sampled permutation in $S_n$. Hence \[
\mathbb{P}(\Lambda_n = \lambda)=\frac{(f^{\lambda})^2}{n!}
\]
Thus the probability that a uniformly random permutation has tableau shape $\lambda$ is $\frac{(f^{\lambda})^2}{n!}$. This is the Plancherel measure on partitions of $n$. By the hook-length formula \citep{hook}, we have \begin{equation}
    f^{\lambda}=\frac{n!}{\prod_{b\in \lambda}h(b)}
\end{equation}
where $b$ is an individual box in the young tableau of shape $\lambda$ and the function $h$ is \begin{equation}
    h(b)=a(b)+l(b)+1
\end{equation}
with $a(b)$ being the number of boxes to the right of $b$ in the same row and $l(b)$ being the number of boxes below $b$ in the same column. Here $a(b)$ and $l(b)$ denote the arm-length and leg-length of the box $b$, respectively, and $h(b)=a(b)+l(b)+1$ is the hook length. Since we showed above that the number of passes in the DS sequence equals $\lambda_1'$ we obtain \begin{equation}
\label{eq:EKnplanch}
    \mathbb{E}[D_n] = \sum_{\lambda\vdash n}\lambda_1'\mathbb{P}(\Lambda_n=\lambda) = \sum_{\lambda\vdash n}\lambda_1'\frac{(f^{\lambda})^2}{n!}=\mathbb{E}[\Lambda'_{1_n}]
\end{equation}
where $\Lambda'_{1_n}$ denotes the first-column length of $\Lambda_n$. This exact formula is computationally preferable to brute-force enumeration for moderate $n$, since the number of partitions of $n$ is much smaller than $n!$. We visualize this relationship with the following Monte Carlo simulation in Figure \ref{fig:regularMC}.
\begin{figure}
\begin{center}
\includegraphics[width=0.8\textwidth]{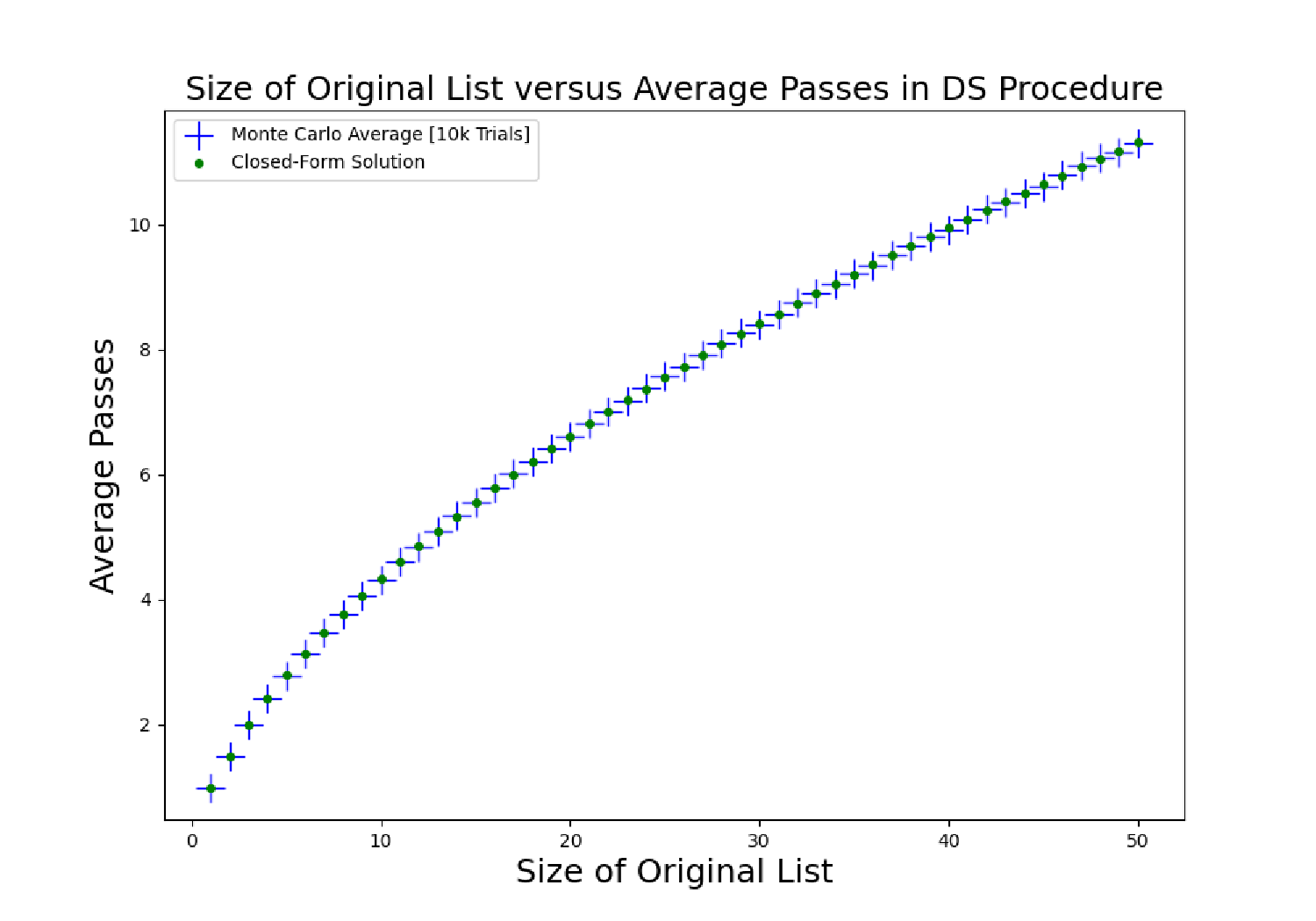}
\caption[Non-resampling Monte Carlo simulation]{A Monte Carlo simulation to visualize the exact formula for the non-resampling DS problem.}
\label{fig:regularMC}
\end{center}
\end{figure}

\subsection{Asymptotic Analysis of the DS Procedure}

We now turn to the asymptotic analysis of the first moment in the non-resampling DS procedure. 
\begin{theorem}
Let $p_n$ be a uniformly random permutation in $S_n$, and let $D_n := D(p_n)$ denote
the number of passes performed by the non-resampling DS procedure on $p_n$.
Let $\Lambda_n \vdash n$ be the random partition given by the shape of the tableau
associated with $p_n$ under the Robinson--Schensted correspondence, and let
$\Lambda_{1_n}$ and $\Lambda'_{1_n}$ denote the lengths of the first row and first column
of $\Lambda_n$, respectively. Then
\[
D_n = \Lambda'_{1_n} \qquad \text{almost surely}
\]
In particular,
\[
\frac{D_n}{2\sqrt{n}} \to 1
\qquad \text{in probability as } n \to \infty,
\]
and
\[
\frac{D_n - 2\sqrt{n}}{n^{1/6}} \Rightarrow \chi_{\mathrm{GUE}},
\]
where $\chi_{\mathrm{GUE}}$ has the Tracy--Widom $F_2$ distribution. Moreover,
\[
\mathbb{E}[D_n] = 2\sqrt{n} + \mu_2 n^{1/6} + o(n^{1/6}),
\]
where
\[
\mu_2 := \mathbb{E}[\chi_{\mathrm{GUE}}],
\]
and in particular
\[
\mathbb{E}[D_n] \sim 2\sqrt{n}.
\]
\end{theorem}

\begin{proof}
By Theorem \ref{thm:DS-LDS}, we have
\[
D(p) = L(p)
\]
for every permutation $p \in S_n$, where $L(p)$ denotes the length of the longest
decreasing subsequence of $p$. Applying this to $p_n$, we obtain
\[
D_n = L(p_n).
\]

Under the Robinson--Schensted correspondence, if $\Lambda_n$ is the shape associated
with $p_n$, then the length of the longest increasing subsequence of $p_n$ is $\Lambda_{1_n}$,
while the length of the longest decreasing subsequence is $\Lambda'_{1_n}$.
Hence
\[
D_n = L(p_n) = \Lambda'_{1_n}
\]
almost surely.

Since the Plancherel measure is invariant under conjugation of partitions, $\Lambda_n$ and
its conjugate $\Lambda_n'$ have the same distribution. Therefore
\[
\Lambda'_{1_n} \stackrel{d}{=} \Lambda_{1_n},
\]
and so
\[
D_n \stackrel{d}{=} \Lambda_{1_n}.
\]

The classical results of \citet{loganshepp} and \citet{vershik} imply that
\[
\frac{\Lambda_{1_n}}{2\sqrt{n}} \to 1
\qquad \text{in probability,}
\]
and therefore
\[
\frac{D_n}{2\sqrt{n}} \to 1
\qquad \text{in probability.}
\]

Moreover, the Baik--Deift--Johansson theorem gives \citep{bdj}
\[
\frac{\Lambda_{1_n} - 2\sqrt{n}}{n^{1/6}} \Rightarrow \chi_{\mathrm{GUE}}.
\]
Since $D_n \stackrel{d}{=} \Lambda_{1_n}$, it follows that
\[
\frac{D_n - 2\sqrt{n}}{n^{1/6}} \Rightarrow \chi_{\mathrm{GUE}}.
\]

Finally, since $D_n$ and $\Lambda_{1_n}$ have the same distribution, they have the same
expectation. Thus
\[
\mathbb{E}[D_n] = \mathbb{E}[\Lambda_{1_n}].
\]
Using the classical asymptotic expansion for the Plancherel first row \citep{bdj},
\[
\mathbb{E}[\Lambda_{1_n}] = 2\sqrt{n} + \mu_2 n^{1/6} + o(n^{1/6}),
\]
where $\mu_2 = \mathbb{E}[\chi_{\mathrm{GUE}}]$, we conclude that
\[
\mathbb{E}[D_n] = 2\sqrt{n} + \mu_2 n^{1/6} + o(n^{1/6}),
\]
and hence
\[
\mathbb{E}[D_n] \sim 2\sqrt{n}.
\]
\end{proof}

\section{Time Analysis}\label{sec:ta}
In this section we describe an $O(n\log n)$ algorithm for computing the DS score of a sequence. By Theorem \ref{thm:DS-LDS}, the number of passes needed to eliminate a sequence is exactly the length of its longest decreasing subsequence. A direct implementation of the non-resampling DS procedure has worst-case running time $O(n^2)$ so it is preferable to work through the subsequence characterization. For a sequence $p$ let $LDS(p)$ denote the length of its longest decreasing subsequence and let $LIS(p)$ denote the length of its longest increasing subsequence. Using the classical identity $LIS(-p)=LDS(p)$ \citep{LISLDS}, it therefore suffices to compute the length of the longest increasing subsequence of $-p$, which we do in Algorithm \ref{alg:tails}. 

Algorithm \ref{alg:tails} scans the sequence once, and for each element performs a binary search in the usual tail array that tracks the smallest possible terminal values of increasing subsequences of each length. This gives us a time complexity of $O(n\log n)$. Although Algorithm \ref{alg:tails} is the standard method for computing LIS length and is not new, its connection with the DS procedure provides a more efficient way to compute the DS score, using the $LIS(-p)=LDS(p)$ identity.

\begin{algorithm}[H]
\caption{DS Score via LIS of the Negated Sequence}
\label{alg:tails}
\begin{algorithmic}[1]
\Require Sequence $p=(p_1,\ldots,p_n)$
\Ensure $D(p)=\operatorname{LDS}(p)$
\State $T \gets [\ ]$ \Comment{$T[\ell]$ = minimum possible tail of an increasing subsequence of length $\ell+1$}
\For{$k \gets 1$ to $n$}
    \State $x \gets -p_k$
    \State $i \gets \operatorname{LowerBound}(T,x)$
    \Comment{smallest index $i$ with $T[i] \ge x$; returns $\mathrm{len}(T)$ if none}
    \If{$i=\operatorname{len}(T)$}
        \State append $x$ to $T$
    \Else
        \State $T[i] \gets x$
    \EndIf
\EndFor
\State \Return $\operatorname{len}(T)$
\end{algorithmic}
\end{algorithm}

\newpage

\bibliographystyle{plainnat}
\bibliography{ref_clean.bib}
\end{document}